\documentclass{amsart}

\usepackage{amssymb, amsthm, stmaryrd}
\usepackage{amsmath, enumerate, tikz, float}

\def\nor{\trianglelefteq}
\def\iso{\cong}
\def\ord#1{\vert #1 \vert}
\def\ind#1#2{\vert #1\, :\,#2\vert}
\def\ZZ{\mathbb{Z}}
\def\Z#1{\textrm{Z}(#1)}
\def\CD#1{\mathcal{C}\mathcal{D} (#1)}

\newtheorem{thm}{Theorem}
\newtheorem{prop}[thm]{Proposition}
\newtheorem{cor}[thm]{Corollary}

\title{Quasi-antichain Chermak-Delgado Lattices of Finite Groups}

\author{Ben Brewster}
\address{Ben Brewster, Department of Mathematical Sciences, Binghamton University, Binghamton, New York, 13905}
\email{ben@math.binghamton.edu}

\author{Peter Hauck}
\address{Peter Hauck, Fachbereich Informatik, Eberhard-Karls-Universit\"at T\"ubingen, T\"ubingen, Germany}
\email{hauck@informatik.uni-tuebingen.edu}

\author{Elizabeth Wilcox}
\address{Elizabeth Wilcox, Mathematics Department, Oswego State University, Oswego, New York, 13126}
\email{elizabeth.wilcox@oswego.edu}

\dedicatory{Dedicated to Otto H. Kegel for the occasion of his eightieth birthday.}

\begin{document}

\begin{abstract} The Chermak-Delgado lattice of a finite group is a dual, modular sublattice of the subgroup lattice of the group.  This paper considers groups with a quasi-antichain interval in the Chermak-Delgado lattice, ultimately proving that if there is a quasi-antichain interval between subgroups $L$ and $H$ with $L \leq H$ then there exists a prime $p$ such that $H / L$ is an elementary abelian $p$-group and the number of atoms in the quasi-antichain is one more than a power of $p$.  In the case where the Chermak-Delgado lattice of the entire group is a quasi-antichain, the relationship between the number of abelian atoms and the prime $p$ is examined; additionally, several examples of groups with a quasi-antichain Chermak-Delgado lattice are constructed. \end{abstract}

\maketitle

This paper pursues the nature of the Chermak-Delgado lattice of a finite group.  The Chermak-Delgado lattice was introduced by Chermak and Delgado \cite{cd}.  Isaacs \cite{Isaacs} re-introduced the lattice, sparking further study that resulted in \cite{BW2012} and \cite{BHW2013}.  In this article, we provide three primary contributions to the study of Chermak-Delgado lattices: a description of the structure of groups with a quasi-antichain (defined below) as an interval in the Chermak-Delgado lattice, results that narrow the possible structure of a quasi-antichain realized as a Chermak-Delgado lattice, and examples to illustrate the breadth of possibilities. Among these contributions is a proof that if a Chermak-Delgado lattice has an interval which is a quasi-antichain then the width must be a power of a prime plus 1.

Throughout the article, let $G$ be a finite group and $p$ be a prime.  The Chermak-Delgado lattice of a finite group $G$ consists of subgroups $H \leq G$ such that $\ord H \ord {C_G(H)}$ is maximal among all subgroups of $G$.  For any subgroup $H$ of $G$, the product $\ord H \ord {C_G(H)}$ is called the {\it Chermak-Delgado measure of $H$} (in $G$) and is denoted by $m_G(H)$; if the group $G$ is clear from context then we write simply $m(H)$. To denote the maximum possible Chermak-Delgado measure of a subgroup in $G$ we write $m^*(G)$ and we refer to the set of all subgroups with measure attaining that maximum as the Chermak-Delgado lattice of $G$, or $\CD G$.

It is known that the Chermak-Delgado lattice is a modular self-dual lattice and if $H, K \in \CD G$ then $HK = KH = \langle H, K \rangle$.  The duality of the Chermak-Delgado lattice is a result of the fact that if $H \in \CD G$ then $C_G(H) \in \CD G$ and $H = C_G (C_G(H))$.  Moreover, if $M$ is the maximum subgroup in the Chermak-Delgado lattice of a group $G$ then the Chermak-Delgado lattices of $G$ and $M$ coincide.  It is additionally known that the co-atoms in the Chermak-Delgado lattice are normal in $M$ and consequently the atoms, as centralizers of normal subgroups, are also normal in $M$.  In \cite{BHW2013}, groups whose Chermak-Delgado lattice is a chain were studied; a {\it chain of length $n$}, where $n$ is a positive integer, is a totally ordered lattice with $n+1$ subgroups. We call a lattice consisting of a maximum, a minimum, and the atoms of the lattice a {\it quasi-antichain} and the width of the quasi-antichain is the number of atoms.  A quasi-antichain of width 1 is also a chain of length 2.

Let $L \leq H \leq G$; we use $\llbracket L, H \rrbracket$ to denote the interval from $L$ to $H$ in a sublattice of the lattice of subgroups of $G$.  If $\llbracket L, H \rrbracket$ is an interval in $\CD G$ then the duality of the Chermak-Delgado lattice tells us that $\llbracket C_G(H), C_G(L) \rrbracket$ is an interval in $\CD G$.  Of course, these intervals may overlap or even coincide exactly. In Section~\ref{intervals} we make no assumption about the intersection of $\llbracket L, H \rrbracket$ and $\llbracket C_G(H), C_G(L) \rrbracket$; in the final two sections we study the situation where these two intervals not only are equal, but are the entirety of $\CD G$.

\section{Quasi-antichain Intervals in Chermak-Delgado Lattices}\label{intervals}

Let $G$ be a group with $L < H \leq G$ such that $\llbracket L, H \rrbracket$ is an interval in $\CD G$.  The main theorem of this section establishes that if $\llbracket L, H \rrbracket$ is a quasi-antichain of width $w \geq 3$ then there exists a prime $p$ and positive integers $a,b$ with $b \leq a$ such that $H / L$ is an elementary abelian $p$-group with order $p^{2a}$ and $w = p^b + 1$.  To make the role of the duality more transparent in the proofs, set $H^* = C_G(L)$ and $L^* = C_G(H)$.  Observe that $C_G(H^*) = L$ and $C_G(L^*) = H$.  

We start with a general statement about subgroups that are in the interval $\llbracket L, H \rrbracket$ in $\CD G$.  Let $p$ be a prime dividing the index of $L$ in $H$.    For this result, we remind the reader that the notation $\Omega_k(M)$, where $k$ is a positive integer and $M$ is any group, denotes the subgroup of $M$ generated by the elements whose order divides $p^k$.  

The hypothesis of Proposition~\ref{omegaspt1} may initially sound restrictive: We require that $G$ be a group with $\llbracket L, H \rrbracket$ in $\CD G$ such that $[HH^*, HH^*] \leq L \cap L^*$. Note that $L \nor H$ and $L^* \nor H^*$, so the quotient groups described in Proposition~\ref{omegaspt1} are well-defined.   Moreover, notice that the hypotheses of the proposition occur when $G $ is a $p$-group of nilpotence class 2 and $H = G \in \CD G$.  

\begin{prop}\label{omegaspt1} Let $G$ be a group with an interval $\llbracket L, H \rrbracket$ in $\CD G$ such that $[HH^*, HH^*] \leq L \cap L^*$.  Suppose that $p$ is a prime dividing $\ord {H / L}$.  The subgroups $A_k(H), B_k(H) \leq H$ where $A_k(H) / L = \Omega_k (H / L)$ and $B_k(H) = \langle x^{p^k} \mid x \in H \rangle L$ are members of $\CD G$ for all positive values of $k$, as are the similarly defined subgroups $A_k(H^*),B_k(H^*)$ of $H^*$. \end{prop}

\begin{proof} Let $k$ be a positive integer.  Without loss of generality, assume that $\ord {A_k(H)/L} \geq \ord {A_k(H^*)/L^*}$.  We first show that $C_G(A_k(H)) = B_k(H^*)$ and that $A_k(H) \in \CD G$.   

Observe that if $x \in H$ and $y \in H^*$ then $[x,y] \in [H,H^*] \leq L \cap L^* = C_G(H^*) \cap C_G(H)$.  Therefore $[x^p,y] = [x,y^p]$ whenever $x \in H$ and $y \in H^*$.  Moreover, if $x \in A_k(H)$ then $x^{p^k} \in L$, so $[x^{p^k}, y] = 1$ for all $y \in C_G(L) = H^*$.  Thus if $x \in A_k(H)$ and $y$ is a generator of $B_k(H^*)$ then $[x,y] = 1$, therefore $B_k(H^*) \leq C_G(A_k(H))$.

Since the quotient $H / L$ is abelian, $\ord {A_k(H) / L} = \ord {H / B_k(H)}$ or equivalently $\ord {A_k(H) / L}  \ord {B_k(H)} = \ord H$. The same is true regarding $\ord {H^*}$ and the subgroups $A_k(H^*)$, $B_k(H^*)$.  Thus the measure of $A_k(H)$ in $G$ can be calculated as follows:
$$
\begin{array}{rcl}
m(A_k(H)) = \ord {A_k(H)} \ord {C_G(A_k(H))} &\geq& \ord {A_k(H)} \ord {B_k(H^*)}\\
&=& \ord {A_k(H) / L} \ord L \ord {B_k(H^*)}\\
&\geq& \ord {A_k(H^*) / L^*} \ord L \ord {B_k(H^*)}\\
&=& \ord {H^*} \ord L\\
&=& \ord {C_G(L)} \ord L = m^*(G).\\
\end{array}
$$
Therefore each inequality above is actually an equality, with $C_G(A_k(H)) = B_k(H^*)$ and $\ord {A_k(H)/L} = \ord {A_k(H^*)/L^*}$. Additionally $A_k(H)$, $B_k(H^*)$, $A_k(H^*)$, $B_k(H) \in \CD G$.\end{proof}

For the rest of the paper, we study intervals that are quasi-antichains.  Ultimately we will use Proposition~\ref{omegaspt1} to show that $A_1 = H$ and $A_1^* = H^*$ in the case that $\llbracket L, H \rrbracket$ is a quasi-antichain of width $w \geq 3$.  The next two propositions establish important facts about the atoms of a quasi-antichain interval in $\CD G$, as well as show that such an interval satisfies the hypothesis of Proposition~\ref{omegaspt1}. 

Let $\llbracket L, H \rrbracket$ be a quasi-antichain of width $w$ throughout the remainder of the article.  Let the $w$ atoms of the quasi-antichain be denoted by $K_1, K_2, \dots, K_w$.  The interval $\llbracket L^*, H^* \rrbracket$ is also a quasi-antichain in $\CD G$, with atoms $C_G(K_i)$ where $1 \leq i \leq w$.  For each $i$, let $K_i^* = C_G(K_i)$ so that $C_G(K_i^*) = K_i$.

\begin{prop}\label{basic}  If $K_1$, $K_2$ are distinct atoms of the quasi-antichain then $K_i \nor H$ for $i = 1,2$, $L \nor H$, and $[K_1,K_2] \leq L$ and analogously for $H^*$, $K_1^*$, $K_2^*$, and $L^*$.  Moreover, $\ord {K_1: L} = \ord {K_2^*: L^*}$.  

If $w \geq 3$ then $K_i / L \iso K_j / L$ and $K_i^* / L^* \iso K_j^* / L^*$ for all $i,j$ with $1 \leq i,j \leq w$. Furthermore: 
$$
\ord {H / L} = {\ord{H / K_1}}^2 = {\ord {H^* / K_1^*}}^2 = \ord {H^*/ L^*}.
$$
\end{prop}

\begin{proof} Let $K_1, K_2 \in \CD G$ with $L < K_i < H$ for $i = 1,2$. Because the interval $\llbracket L,H \rrbracket $ is a quasi-antichain, $H = K_1K_2$ and $K_1 \cap K_2 = L$.  From this structure and because $H$ cannot equal $K_1K_1^h$ for $h \in H$, it follows that $K_1 \nor H$ (similarly for $K_2$).   Therefore $L \nor H$ and $[K_1, K_2] \leq L$.  The equality $m^*(G) = m(H) = m(K_2)$ implies
$$
\frac{\ord {K_2} \ord {K_1}}{\ord L} \ord {C_G(H)} = \ord H \ord {C_G(H)} = \ord {K_2} \ord {C_G(K_2)},
$$
and consequently $\ind {K_1}{L} = \ind {C_G(K_2)}{C_G(H)} = \ind {K_2^*}{L^*}$. 

In the situation that $w \geq 3$, then $H = K_1K_2 = K_3K_2$ where $K_1 \cap K_2 = K_2 \cap K_3 = L$ and thus $\ord {K_i} = \ord {K_j}$ for all $i, j$ with $1 \leq i, j \leq w$.  This additionally yields $K_i / L \iso K_j / L$ for all $i, j$.  From the Isomorphism Theorems it follows that $\ord {H / K_1} = \ord {K_1 / L}^2$.  

The same arguments applied to the quasi-antichain $\llbracket L^*, H^* \rrbracket$ yield the remaining assertions. \end{proof}

\begin{prop}\label{centralizing} If $w \geq 3$ then $[H,H^*] \leq L \cap L^* = C_G(HH^*)$.  Additionally, $H / L$ and $H^* / L^*$ are isomorphic elementary abelian $p$-groups. In particular, if $G = H$ and $G \in \CD G$ then $G / \Z G$ and $[G,G]$ are elementary abelian $p$-groups. \end{prop}

\begin{proof} Since $w \geq 3$, there exist at least three distinct atoms $K_1$, $K_2$, and $K_3$ in the interval $\llbracket L,H \rrbracket$ in $\CD G$.  By Proposition~\ref{basic},
$$
[K_1, K_2K_3] \leq \langle [K_1,K_3][K_1,K_2] \rangle \leq L.
$$
Therefore $K_1 / L$ centralizes $K_2K_3 / L = H / L$. By symmetry, the same holds for $K_2 / L$; consequently $H = K_1K_2$ centralizes $H / L$ and $H / L$ is abelian. Similarly $H^*$ centralizes $H^* / L^*$ and the latter is abelian.

Since $K_i$ normalizes every $K_j$, it also normalizes every $K_j^*$. Therefore $[K_i,H^*] = [K_i,K_1^*K_2^*] \leq K_2^*$ and $[K_i,H^*] = [K_i,K_1^*K_3^*] \leq K_3^*$, so that $[K_i,H^*] \leq L^*$.  Similarly $[K_2,H^*] \leq L^*$ and thus $[H,H^*] \leq L^*$.  In the same way, $[H,H^*] \leq L$.  By the duality of the Chermak-Delgado lattice, $L \cap L^* = C_G(HH^*)$ so $[H,H^*] \leq L \cap L^* = C_G(HH^*)$, as claimed.

Applying Proposition~\ref{omegaspt1}, the subgroup $A_1$ where $A_1 / L = \Omega_1(H/L)$ is a member of $\CD G$.  Since $\CD G$ is a quasi-antichain, either $A = H$ or there exists $i$ such that $1 \leq i \leq w$ where $K_i = A$.  At minimum, $K_i / L$ is an elementary abelian $p$-group but, as $K_i / L \iso K_j / L$ for all $i, j$, we have $H / L$ is an elementary abelian $p$-group.  With similar reasoning, $H^* / L^*$ is an elementary abelian $p$-group and, since $\ord {H / L} = \ord {H^* / L^*}$, these quotients are isomorphic elementary abelian $p$-groups. 

If $H = G$ and $\Z G = L$ then $G / \Z G$ is an elementary abelian $p$-group.  Thus, for $x, y \in G$, we have $[x,y]^p = [x^p,y] = 1$; therefore $[G,G]$ is elementary abelian. \end{proof}

\begin{thm}\label{w} Let $G$ be a group such that $L, H \in \CD G$ with the interval $\llbracket L, H \rrbracket$ in $\CD G$ a quasi-antichain of width $w \geq 3$.  There exists a prime $p$ and positive integers $a, b$ with $b \leq a$ such that $H / L$ and $C_G(L) / C_G(H)$ are elementary abelian $p$-groups of order $p^{2a}$ and $w = p^b + 1$.\end{thm}

\begin{proof} The existence of the prime $p$ and the fact that $H / L$ and $H^* / L^*$ are elementary abelian $p$-groups were established in Proposition~\ref{centralizing}. From Proposition~\ref{basic}, we know $H / L = K_1/L \times K_2 / L$.  Let $i \geq 3$; the subgroup $K_i / L$ projects onto each coordinate under the natural projection maps and intersects each of $K_1 / L$ and $K_2 / L$ trivially.  Thus $K_i / L$ is a subdirect product and there exists an isomorphism $\overline{\beta_i}: K_1 / L \rightarrow K_2 / L$ such that  $K_i / L = \{ (kL) \overline{\beta_i} (kL) \mid k \in K_1\}$.  Choose $\beta_i(k) \in K_2$ with $\overline{\beta_i}(kL) = \beta_i(k)L$; then $K_i / L = \{ k\beta_i(k)L \mid k \in K_1\}$.  Similarly, there exists an isomorphism $\overline{\alpha_i}: K_1^* / L^* \rightarrow K_2^* / L^*$ where $\overline{\alpha_i}(mL^*) = \alpha_i(m) L^*$ for each $m \in K_1^*$ and $K_i^*/L^* = \{ m\alpha_i(m)L^* \mid m \in K_1^* \}$. 

For $i,j$ such that $3 \leq i,j \leq w$, let $\Delta_{i,j} = \{ k \beta_i (k) \beta_j (k) \mid k \in K_1 \}$ and $\Delta_{i,j}^* = \{ m \alpha_i (m) \alpha_j (m) \mid m \in K_1^* \}$; additionally define $K_{i,j} = \Delta_{i,j} L$ and $K_{i,j}^* = \Delta_{i,j}^* L^*$.  Since $[K_1,K_2] \leq L$ and the functions $\overline{\beta_i}$, $\overline{\beta_j}$, are homomorphisms, it follows that $K_{i,j} \leq H$.  Also observe that if $k\beta_i(k)\beta_j(k)L = k' \beta_i(k') \beta_j(k')L$ then $kL = k'L$, because $K_1 \cap K_2 = L$.  Therefore $\ord {K_{i,j} / L} = \ord {K_1/L}$.  Corresponding facts are true regarding $K_{i,j}^*$.  

Our goal is to show that $K_{i,j}$ is one of the atoms in $\llbracket L, H \rrbracket$, so we calculate $m(K_{i,j})$.  From the definitions, clearly $[k_1,m_1] = [k_2,m_2] = 1$ when $k_i \in K_i$ and $m_i \in K_i^*$ for $i = 1, 2$.  By this information and the fact that $[H,H^*]$ is centralized by $H$ and $H^*$, if $k \in K_1$ and $m \in K_1^*$ then we obtain 
$$
1 = [k\beta_i(k),m\alpha_i(m)] = [k,\alpha_i(m)][\beta_i(k),m]
$$
for all $i$ such that $3 \leq i \leq w$.  Given $k \in K_1$ and $m \in K_1^*$, if $3 \leq i,j \leq w$ then
$$
\begin{array}{l}
[k  \beta_i(k) \beta_j(k), m \alpha_i(m) \alpha_j(m)]\\
\quad = [k,\alpha_j(m)][\beta_i(k),\alpha_j(m)][\beta_j(k),\alpha_i(m)][\beta_j(k),m]\\
\quad = [k,\alpha_j(m)] [\beta_j(k),m]\\
\quad = 1.\\
\end{array}
$$
By $[H,L^*] = [H^*, L] = 1$ and the above calculation, $K_{i,j}^* \leq C_G(K_{i,j})$. Because $\ord {K_{i,j}} = \ord {K_1}$ and $\ord {K_{i,j}^*} = \ord {K_1^*}$, therefore $m(K_{i,j}) = m(K_1)$ and $K_{i,j} \in \CD G$.  Thus for each $i,j$ with $3 \leq i,j \leq w$, either $K_{i,j} = K_h$ for some $h$ such that $3 \leq h \leq w$ or $K_{i,j} = K_1$. Setting $\beta_2(k) = 1$ for all $k \in K_1$, it follows that $\{ k \beta_i(k) \beta_j(k) L \mid k \in K_1 \} = \{ k \beta_h (k) L \mid k \in K_1 \}$ for some $h$ with $2 \leq h \leq w$.  Notice that if $k \beta_i(k) \beta_j(k)L = k' \beta_h(k')L$ then $kL = k'L$ and $\beta_i(k)\beta_j(k)L = \beta_h(k)L$, because $K_1 \cap K_2 = L$.

Fix a $k_1 \in K_1 \setminus L$.  Let $\Lambda = \{\beta_2(k_1), \beta_3(k_1), \dots \beta_w(k_1) \}$ and $R = \Lambda \cdot L$. By what we have shown in the preceding paragraphs, $R \leq H$ and, as $K_i \cap K_j = L$ for $i \ne j$, the set $\Lambda$ is a transversal for $L$ in $R$. Hence $\ord {R / L} = \ord {\Lambda} = w - 1$.  Since $R \leq K_2$, it follows that $w-1$ divides $p^a$. \end{proof}

\section{Quasi-antichain Chermak-Delgado Lattices}\label{antichains}

We study here the groups $G$ such that $\CD G$ is a quasi-antichain and $G \in \CD G$, meaning that $G = H = H^*$ and $\Z G = L = L^*$ in the notation of the first section.  Additionally, the subgroups $K_i^*$ are now atoms of $\llbracket L, H \rrbracket$; notice $K_i^* = K_i$ if and only if $K_i$ is abelian. 

When studying groups of this type, the added condition that $[G,G]$ be cyclic imposes very strong restrictions on the structure of the group, as seen in the next proposition.  

\begin{prop}\label{xspec} Let $G \in \CD G$ and $[G,G]$ be cyclic. Then $\CD G$ is a quasi-antichain of width $w \geq 3$ with $G \in \CD G$ if and only if there exists a prime $p$ such that $\ord {[G,G]} = p$ and $G / \Z G \iso C_p \times C_p$. In this case $w = p + 1$. \end{prop}

\begin{proof} Let $[G,G]$ is cyclic and $G \in \CD G$.   Suppose first that $\CD G$ is a quasi-antichain of width $w \geq 3$.  We know that there exists a prime $p$ such that $G / \Z G$ and $[G,G]$ are elementary abelian $p$-groups by Proposition~\ref{centralizing}.  Therefore $[G,G]$ has order $p$ but, more importantly, all $U \leq G$ such that $\Z G \leq U$ are centralizers by \cite[Satz]{reuther77}.  In particular, a maximal subgroup $M < G$ is a centralizer so there exists $U > \Z P$ with $M = C_G(U)$ and $m(M) = \frac{\ord G}{p} \ord {C_G(X)} \geq \ord {G} \ord {\Z G}$.  Yet $G \in \CD G$, so $M, U \in \CD G$.  The Chermak-Delgado lattice of $G$ is a quasi-antichain of width at least 3 so by Proposition~\ref{basic}, $\ord M = \ord U$ and thus $\ord {G / \Z P} = p^2$.

Now suppose there exists a prime $p$ such that $\ord {[G,G]} = p$ and $G / \Z G \iso C_p \times C_p$.  In this case, all $p + 1$ subgroups $U$ such that $\Z G < U < G$ are abelian, have order $p \ord {\Z G}$, and have measure $p^2 {\ord {\Z G}}^2$, which also equals the measure of $G$.  Therefore $\CD G = \{ U \leq G \mid \Z G \leq U\}$ is a quasi-antichain of width $p + 1$ with $G \in \CD G$.  \end{proof}

The next theorem justifies our attention on $p$-groups while studying groups with a quasi-antichain Chermak-Delgado lattice.  The proof of Theorem~\ref{pgrp} requires the observation: Let $M$ and $N$ be any pair of finite groups. The modularity of the Chermak-Delgado lattice implies that every maximal chain in the lattice has the same length.  For example, all maximal chains in $\CD G$ have  length 2 because $\CD G$ is a quasi-antichain.  That $\CD {M \times N} \iso \CD M \times \CD N$ as lattices \cite{BW2012} gives that the length of a maximal chain in $M \times N$ is the sum of the lengths of maximal chains in $M$ and $N$.  

\begin{thm}\label{pgrp} If $G$ is a group with $\CD {G}$ a quasi-antichain of width $w \geq 3$ and $G \in \CD G$ then $G$ is nilpotent of class 2; in fact, there exists a prime $p$, a nonabelian Sylow $p$-subgroup $P$ with nilpotence class 2, and an abelian Hall $p'$-subgroup $Q$ such that $G = P \times Q$, $P \in \CD P$, and $\CD G \iso \CD P$ as lattices. Moreover there exist positive integers $a,b$ with $b \leq a$ such that $\ord {G / \Z G} = \ord {P / \Z P} = p^{2a}$ and $w = p^b + 1$. \end{thm}

\begin{proof} Note that $G$ is nilpotent, by Proposition~\ref{centralizing}, but nonabelian and the length of a maximal chain in $\CD G$ is 2. If $G = Q_1 \times Q_2$ where $Q_1$ and $Q_2$ are Hall $\pi$-, $\pi'$-subgroups of $G$, respectively, then $\CD G \iso \CD {Q_1 \times Q_2}$; as a consequence of the additivity of chain length over a direct product, if both $Q_1$ and $Q_2$ are nonabelian then $\CD {Q_i} = \{Q_i, \Z {Q_i}\}$ for $i = 1, 2$. However, this implies $\CD {Q_1 \times Q_2}$ is a quasi-antichain of width 2.  Consequently, exactly one of $Q_1$ or $Q_2$ is abelian and the Chermak-Delgado lattice of the nonabelian factor is isomorphic (as lattices) to $\CD G$.  Therefore there exists a unique prime $p$ such that $G = P \times Q$ where $P$ is a nonabelian Sylow $p$-subgroup of $G$ and $Q$ is an abelian Hall $p'$-subgroup of $G$, with $\CD G \iso \CD P$ as lattices. The rest follows from Proposition~\ref{basic} and Theorem~\ref{w}.\end{proof}

We investigated the number of abelian atoms that is permitted in a quasi-antichain Chermak-Delgado lattice.  The final theorem of this section records our contributions in this direction.

\begin{thm}\label{s>1} Let $G$ be a $p$-group with $\CD G$ a quasi-antichain of width $w \geq 3$ and suppose $\ord {G / \Z G} = p^{2a}$ for a positive integer $a$.  Let $t$ be the number of abelian atoms in $\CD G$ and $u$ be the number of pairs of nonabelian atoms. 
	\begin{enumerate}
	\item If $t = 0$ then $p$ is odd.
	\item If $t = 1$ then $p = 2$.
	\item If $t \geq 2$ then there exists a positive integer $c \leq a$ such that $t = p^c + 1$. In particular, $p - 1$ divides $t - 2$; if $p$ is odd then $p^c$ divides $u$ and if $p = 2$ then $t \geq 3$ and $2^{c-1}$ divides $u$.
	\item If $t \geq 2$ and $u \geq 1$ then $3 \leq t \leq 2u + 1$ when $p = 2$ and $2 \leq t \leq u + 1$ when $p$ is odd.
	\item If $t \geq 3$ then $t \geq p + 1$.
	\end{enumerate}\end{thm}

\begin{proof} Theorem~\ref{w} tells us that $w = p^b + 1$ for some positive integer $b \leq a$, but also $w = t + 2u$ as set up by the notation.  If $t = 0$ then $w = 2u = p^b + 1$, necessarily forcing $p$ to be odd.  If $t = 1$ then $2u = p^b$; clearly $p$ must equal 2 in this case.  

Suppose that $t \geq 2$; we continue here with the same notation and set up as in the proof of Theorem~\ref{w} except we add the condition that $K_1$ and $K_2$ are abelian atoms. Recall the fixed $k_1 \in K_1 \setminus L$ and that $\beta_2(k) = 1$ for all $k \in K_1$.  Set $\Gamma = \{ \beta_i(k_1) \mid 2 \leq i \leq w \textrm{ and } K_i = K_i^* \}$, a subset of $\Lambda$.  

Let $i, j \geq 3$ and assume that $K_i$ and $K_j$ are abelian atoms; we show that $K_{i,j}$ is also abelian.  We use the functions $\alpha_i$ defined in the proof of Theorem~\ref{w}.  Observe that $\alpha_i(k)$ now differs from $\beta_i(k)$ only by an element in $\Z G$, for all $k \in K_1$. The calculation below follows:
$$
[k\beta_i(k)\beta_j(k),k'\beta_i(k')\beta_j(k')] = [k\beta_i(k)\beta_j(k),k'\alpha_i(k')\alpha_j(k') ] = 1
$$
for all $k, k' \in K_1$. Therefore $K_{i,j}$ is also an abelian atom in $\CD G$. Since $K_{i,j} \ne K_2$, it follows that $\beta_(k_1)\beta_j(k_1)\Z G = \beta_h(k_1)\Z G$ for some $\beta_h(k_1) \in \Gamma$.  Consequently $\Gamma \Z G \leq K_2$ and since $\Gamma \subseteq \Lambda$, the elements of $\Gamma$ are distinct.  Therefore $\ord {\Gamma \Z G / \Z G} = \ord {\Gamma} = t-1$ divides $p^a$ and there exists a positive integer $c$ such that $t - 1 = p^c$. 

Since $t = p^c + 1$, clearly if $p = 2$ and $t \geq 2$ then $t = 2^c + 1 \geq 3$ but, for all primes $p$, it is true that $p - 1$ divides $p^c - 1 = t - 2$.  Observe, for part (5), that if $t \geq 3$ then $t \geq p + 1$ is necessary for $t - 2$ to be a multiple of $p - 1$.  To complete the proof of part (3), notice that $2u = w - t = p^c(p^{b-c} - 1)$, so that $u = \frac{1}{2} p^c (p^{b-c} - 1)$.  If $p = 2$ then $2^{c-1}$ divides $u$, and if $p$ is odd then $u$ must be divisible by $p^c$.  This completes the assertions in part (3).

Continuing with part (4), suppose that $t \geq 2$ and $u \geq 1$.  Then $b > c$, so $p^{b-c} - 1> p - 1$.  If $p$ is odd, this implies $\frac{1}{2} (p^{b-c} - 1) \geq 1$ and consequently 
$$
t = p^c + 1 \leq p^c \left(\frac{p^{b-c} - 1}{2}\right) + 1 = u + 1.
$$
If $p = 2$ then $2^{b-c} - 1 \geq 1$, so that 
$$
t = 2^c + 1 \leq 2^c(2^{b-c} = 1) + 1 = 2u + 1.
$$
Thus when $t \geq 2$ and $u \geq 1$, the inequalities asserted in part (4) of the theorem are true. \end{proof}

\begin{cor} While there exist finite groups with Chermak-Delgado lattice a quasi-antichain of width 6, there does not exist such a group with exactly 4 abelian atoms in its Chermak-Delgado lattice. \end{cor}

\begin{proof} An extraspecial group of order $5^3$ has a Chermak-Delgado lattice that is a quasi-antichain of width 6 by Proposition~\ref{xspec}.  If we assume that $G$ is a finite group with $\CD G$ a quasi-antichain of width 6 having exactly 4 abelian atoms then we know that $G$ has a Sylow $5$-subgroup $P$ with $\CD G \iso \CD P$ as lattices by Proposition~\ref{pgrp} and Theorem~\ref{w}.  Theorem~\ref{s>1} forces $4 = t \leq u + 1 = 2$; thus $G$ cannot exist. \end{proof}

\section{Examples}

In this section we construct several examples of $p$-groups having a quasi-antichain for their Chermak-Delgado lattice. The first two examples show that every possible quasi-antichain of width 2 can be realized as the Chermak-Delgado lattice of a $p$-group.

\begin{enumerate}
\item {\bf A group $G$ where $\CD G$ is a quasi-antichain of width 2 with no abelian atoms}: Let $H$ be any group with $\CD H = \{H, \Z H\}$.  A family of $p$-groups, each member of which having such a Chermak-Delgado lattice, was described in \cite{BHW2013}. 

Define $G = H \times H$.  In \cite{BW2012} it was established that $\CD {G}$ is a quasi-antichain of width 2 with atoms $\Z H \times H$ and $H \times \Z H$. Clearly $H \times \Z H = C_G(\Z H \times H)$.  

\item {\bf A group $P$ such that $\CD P$ is a quasi-antichain of width 2 with both atoms abelian}:  Let $P = \langle m_1, m_2, n_1, n_2 \rangle$ where each element has order $p$ and 
$$
[m_1,m_2] = [n_1,n_2] = 1, \quad [m_i, n_j] = z_{ij} \in \Z P \textrm{ for } i, j \in \{1,2\}, 
$$
and $\Z P = \langle z_{i,j} \mid i,j \in \{1,2\} \rangle$ is elementary abelian of order $p^4$.  Clearly $P$ is nilpotent of class 2 with order $p^8$ and Chermak-Delgado measure $p^{12}$.  Let $M = \langle m_1, m_2 \rangle \Z P$ and $N = \langle n_1, n_2 \rangle \Z P$.  It's a straightforward calculation to show that 
$C_P(m) = M$ whenever $m \in M \setminus \Z P$ and $C_P(n) = N$ whenever $n \in N \setminus \Z P$, whereas $C_P(x) = \langle x \rangle \Z P$ for all $x \in P \setminus (M \cup N)$.  Thus of all subgroups containing $\Z P$, only $M$ and $N$ have the largest measure, which is $p^{12}$.  Since $m_P(M) = m_P(N) = m_P(\Z P) = p^{12}$ and no other subgroups have this measure, $\CD P$ is a quasi-antichain of width 2 (containing $P$) such that both atoms are abelian.
\end{enumerate}

We now show that for every prime $p$ and every positive integer $n$, there exists a $p$-group whose Chermak-Delgado lattice is a quasi-antichain of width $p^n + 1$ with all atoms abelian. 

\begin{prop} Let $p$ be a prime and $n$ a positive integer.  Let $P$ be the group of all $3 \times 3$ lower triangular matrices over $\textrm{\bf GF}(p^n)$ with 1s along the diagonal.  The Chermak-Delgado lattice of $P$ is a quasi-antichain of width $p^n+1$ and all subgroups in the middle antichain are abelian.\end{prop}

\begin{proof} By Exercise 39 in \cite[III.16]{Huppert}, $P$ has exactly $p^n + 1$ abelian subgroups of maximal order equal to $p^{2n}$; these subgroups have measure $p^{4n} = m(P)$. If $x \in P \setminus \Z P$, it is easy to check that $\ord {C_P(x)} = p^{2n}$.  Therefore if $U \in \CD P$ with $\Z P < U < P$ then $\ord {C_P(U)} \leq p^{2n}$ and $\ord {U} = \ord {C_P(C_P(U))} \leq p^{2n}$.  If follows that $m^*(P) = p^{4n}$ and $\ord U = \ord {C_P(U)} = p^{2n}$.  If $U \ne C_P(U)$ then $U \cap C_P(U) = \Z P$ since $U \cap C_P(U) \in \CD P$.  But then for $x \in U \setminus \Z P$ we have $C_P(x) = C_P(U)$ by order considerations and therefore $x \in U \cap C_P(U) = \Z P$, a contradiction.  Thus $U = C_P(U)$ is one of the abelian subgroups of maximal order and the assertion follows. \end{proof}

Extraspecial groups of order $p^3$ are examples where each of the $p + 1$ atoms in the quasi-antichain is abelian (Proposition~\ref{xspec}); the next two propositions construct $p$-groups where the Chermak-Delgado lattice is a quasi-antichain of width $p + 1$ and, depending on the value of $p$ modulo $4$, the number of abelian atoms is either $0$, $1$, or $2$.

\begin{prop}\label{bigex} Given any prime $p$ there exists a group $P$ of order $p^9$ such that $\CD P$ is a quasi-antichain of width $p + 1$.  In this example: if $p = 2$ then exactly one of the three atoms of $\CD P$ is abelian, when $p \equiv 1$ modulo 4 then exactly two of the $p+1$ atoms of $\CD P$ are abelian, and if $p \equiv 3$ modulo 4 then none of the atoms in $\CD P$ are abelian.\end{prop}

\begin{proof} Let $P$ be generated by $\{x_1, x_2, x_3, y_1, y_2, y_3\}$ with defining relationships $x_i^p = y_i^p = 1$ and $[x_i, y_j] = 1$ for all $i,j$ such that $1 \leq i, j \leq 3$, $\Z P = \langle z_{1,2}, z_{1,3}, z_{2,3} \rangle$ is elementary abelian with order $p^3$, and $[x_i,x_j] = [y_i, y_j] = z_{ij}$ for every $i,j$ with $1 \leq i < j \leq 3$. Let $M_0 = \langle x_1, x_2, x_3 \rangle \Z P$ and $M_p = \langle y_1, y_2, y_3 \rangle \Z P$.  For $1 \leq i \leq p - 1$, let $M_i = \langle x_1y_1^i, x_2y_2^i, x_3y_3^i \rangle \Z P$. We show that $\CD P = \{ P, \Z P, M_i \mid 0 \leq i \leq p \}$.

Observe that $P$ is the central product of $M_0$ with $M_p$ and $\Z P = M_0 \cap M_p$.  Additionally $C_P(M_0) = M_p$ and vice versa, yielding $m_P(P) = m_P(M_0) = m_P(M_p) = p^{12}$. It is easy to show that if $x \in M_0 \setminus \Z P$ then $C_{M_0} (x) = \langle x \rangle \Z P$ and $C_P(x) = \langle x \rangle M_p$.  It follows that $C_P(\langle x, y \rangle ) = \langle x, y \rangle \Z P$ whenever $x \in M_0 \setminus \Z P$ and $y \in M_p \setminus \Z P$.  

Let there exist $a_i, b_i \in \ZZ / p \ZZ$ such that $x = x_1^{a_1}x_2^{a_2}x_3^{a_3} \in M_0 \setminus \Z P$ and $y = y_1^{b_1}y_2^{b_2}y_3^{b_3} \in M_p \setminus \Z P$.  Assume that at least one of $a_1, a_2, a_3$ and one of $b_1, b_2, b_3$ are non-zero.  For $x'y'$ with similar structure, $x'y' \in C_P(xy)$ if and only if $[x,x'] = [y',y]$.  Further decomposing the commutators reveals 
$$
[x,x'] = \prod\limits_{1 \leq i < j \leq 3} z_{ij}^{a_ia_{j}' - a_j a_{i}'} \textrm{ and }
[y',y] = \prod\limits_{1 \leq i < j \leq 3} z_{ij}^{b_i'b_{j} - b_j' b_{i}};
$$
Thus $[xy,x'y'] = 1$ if and only if each of the three equations $a_ia_{j}' - a_j a_{i}'  - b_i'b_{j} + b_j'b_{i} = 0$ hold where $1 \leq i < j \leq 3$.  If $(a_1, a_2, a_3)$ and $(b_1, b_2, b_3)$ are not scalar multiples then $\langle x, y, m_1^{b_1}m_2^{b_2}m_3^{b_2}n_1^{a_1}n_2^{a_2}n_3^{a_3} \rangle \Z P = C_P(xy)$.  On the other hand, if there exists $k$ such that $(b_1, b_2, b_3) = k(a_1, a_2, a_3)$ then $C_P(xy) = \langle m, m_in_i^{-k^{-1}} \mid 1 \leq i \leq 3\rangle \Z P = \langle m \rangle M_{-k^{-1}}$. Therefore $C_P(M_k) = M_{-k^{-1}}$ for $1 \leq k \leq p - 1$ and $m_P(M_k) = p^{12}$.

It follows then that $m(U) < m(M_k)$ whenever $\Z P < U < M_k$.  Additionally, if $U \leq P$ and there exist $u_1,u_2 \in U$ where $u_1 \in M_k$ and $u_2 \in M_{k'}$ with $k \ne k'$ then $C_P(U) \leq \Z P$.  Hence $m^*(P) = p^{12}$ and $\CD P = \{P, \Z P, M_k \mid 0 \leq k \leq p\}$.

Since $C_P(M_k) = M_{-k^{-1}}$ for $1 \leq k \leq p - 1$, there exists an abelian atom of $\CD P$ if and only if $p = 2$ or $p \equiv 1$ modulo 4.  When $p = 2$ then $M_1$ is the unique abelian atom and if $p \equiv 1$ modulo 4 then $M_1$ and $M_{p-1}$ are both abelian, but no other atom in $\CD P$ is abelian.  When $p \equiv 3$ modulo 4 then there do not exist any abelian atoms in $\CD P$. \end{proof}

\begin{prop}\label{bigex2} Let $p$ be a prime.  There exists a group $Y$ of order $p^9$ such that $\CD Y$ is a quasi-antichain of width $p + 1$.  In this example, if $p = 2$ then exactly one of the three atoms of $\CD Y$ is abelian and if $p$ is odd then exactly two of the $p + 1$ atoms are abelian.\end{prop}

\begin{proof} Define $P$ as in Proposition~\ref{bigex} except designate that $[n_i,n_j] = z_{ij}^{-1}$.  The same arguments made earlier will now show that $C_P(M_i) = M_{i^{-1}}$ for $i = 1, 2, \dots, p-1$.  This forces $M_1$ and $M_{p-1}$ to be abelian, yet the remaining facts still stand.\end{proof}

In lattice theory \cite[Chapter 1, Section 2]{Birkhoff}, a lattice $\mathcal{L}$ has a duality $\theta : \mathcal{L} \rightarrow \mathcal{L}$ if $\theta$ is a bijection and if $A,B \in \mathcal{L}$ with $A \leq B$ implies $\theta(B) \leq \theta(A)$.  Such a duality need not have order 2 as a function; in the case of the Chermak-Delgado lattice the duality {\it does} have order 2.  In particular, the examples in this section show that all possible types of quasi-antichains of width 4 with duality of order 2 occur as Chermak-Delgado lattices of $3$-groups and those of width 3 with duality of order 2 occur as Chermak-Delgado lattices of $2$-groups.  

This leaves several questions open for investigation, including: Which values of $t$ (in the notation of Theorem~\ref{s>1}) are possible in quasi-antichain Chermak-Delgado lattices of width $w = p^n + 1$ when $n > 1$? That is, which dualities can be realized by the centralizer map?  The first open case is when $w = 5$ and $t = 3$. And, are there examples of groups $G$ with $G \in \CD G$ and $\CD G$ a quasiantichain where $t = 0$ and either $p = 2$ or $p \equiv 1$ modulo $4$?

\section*{Acknowledgements}

The research of the second author was supported by Projecto MTM2010-19938-C03-02, Ministerio de Ciencia e Innovaci\'{o}n, Spain. The research of the third author was supported by a Scholarly and Creative Activity Grant from Oswego State University.

\bibliographystyle{amsplain}
\bibliography{references}

\end{document}